\documentclass[11pt]{amsart}

\usepackage{amsmath}
\usepackage{amssymb}
\usepackage{graphicx}

\newtheorem{theorem}{Theorem}[section]
\newtheorem{corollary}[theorem]{Corollary}
\newtheorem{lemma}[theorem]{Lemma}
\newtheorem{proposition}[theorem]{Proposition}
\theoremstyle{definition}
\newtheorem{definition}[theorem]{Definition}
\newtheorem{example}[theorem]{Example}
\newtheorem{remark}[theorem]{Remark}

\usepackage{url}
\usepackage{amscd}
\usepackage{enumerate}
\usepackage{tikz}
\usepackage[all]{xy}
\def \B{{\mathcal B}}

\def \N{{\mathbb N}}

\def \R{{\mathbb R}}

\def \1{{\mathbb 1}}

\numberwithin{equation}{section}

\begin{document}
\date{\today}
\title[A strong Bishop-Phelps property]{A strong Bishop-Phelps property and a new class of Banach spaces with the property $(A)$ of Lindenstrauss.}
\author{Mohammed Bachir}

\address{Laboratoire SAMM 4543, Université Paris 1 Panthéon-Sorbonne, France.}
\email{Mohammed.Bachir@univ-paris1.fr}

\maketitle

\begin{center} This work is dedicated to Gilles Godefroy.
\end{center}

\begin{abstract} We give a class of bounded closed sets $C$ in a Banach space  satisfying a generalized and stronger form of the Bishop-Phelps property studied by Bourgain in \cite{Bj} for dentable sets. A version of the {\it ``Bishop-Phelps-Bollobás''} theorem will be also given. The density and the residuality of bounded linear operators attaining their maximum on $C$ (known in the literature) will be replaced, for this class of sets, by being the complement of a $\sigma$-porous set. The result of the paper is applicable for both linear operators and non-linear mappings. When we apply our result to subsets (from this class) whose closed convex hull is the closed unit ball, we obtain a new class of Banach spaces involving property $(A)$ introduced by Lindenstrauss.  We also establish that this class of Banach spaces is stable under $\ell_1$-sum when the spaces have a same “modulus''. Applications to norm attaining bounded multilinear mappings and Lipschitz mappings will also be given.
\end{abstract}


\newcommand\sfrac[2]{{#1/#2}}

\newcommand\cont{\operatorname{cont}}
\newcommand\diff{\operatorname{diff}}


{\bf 2020 Mathematics Subject Classiﬁcation:} Primary 46B20, 47L05, 28A05; Secondary 47B48, 26A21.

\section{\bf Introduction}
In this paper we will restrict ourselves to the case of real Banach spaces without treating the complex case in order to avoid technical diﬃculties. However, to the best of our knowledge, there should not be a heavy obstacle to extend the results to the complex case. Let $X$ be a real Banach space,  by $B_X$ and $S_X$ we denote the closed unit ball and the unit sphere of $X$ respectively. By $X^*$ we denote the topological dual of $X$.  If $C\subset X$ is a nonempty set, $\overline{\textnormal{co}}(C)$ denotes the closed convex hull of $C$.  By $\mathcal{L}(X,Y)$, we denote the space of all bounded linear operators from $X$ into the Banach space $Y$ and $\mathcal{F}(X,Y)$ denotes the subspace of $\mathcal{L}(X,Y)$ of all finite-rank operators. We say that an operator $T\in \mathcal{L}(X,Y)$ attains its norm if there exists $x_0\in S_X$ such that $\|T\|=\|T(x_0)\|$. A Banach space $X$ is said to have  the property $(A)$ (introduced by Lindenstrauss), if for every Banach space $Y$ the subset $NA(X,Y)$ of norm attaining bounded linear operators is dense in  $\mathcal{L}(X,Y)$ (see \cite{Aco0} and references  therein for background on this subject, see also \cite{ACKLM}, \cite{Ba}, \cite{CCGMR}, \cite{CGMR}, \cite {CS}, \cite{JMZ}). A bounded subset $C$ of a Banach space $X$ has the Bishop-Phelps property (see Bourgain's paper \cite{Bj} and the recent work in \cite{JMZ}) if, for every Banach space $Y$, the set of those operators in $\mathcal{L}(X, Y )$ for which $\sup \lbrace \|T(x)\| : x \in C\rbrace$ is a maximum, is dense in $\mathcal{L}(X, Y )$. Clearly, a Banach space $X$ has property $(A)$ if and only if $B_X$ has the Bishop-Phelps property. For the study of the {\it Bishop-Phelps-Bollob\'as property} we refer to the work of M. D. Acosta, R. M. Aron, D. Garc\'ia and M. Maestre  in \cite{AAGM} (see also \cite{Aco}) S. Dantas, D. Garc\'ia, M. Maestre and \'O. Rold\'an in \cite{DGMR} and the work of S. K Kim and  H.J Lee in \cite{KL} (see also \cite{DGKLM}).

 In \cite{Lj}, Lindenstrauss introduced the following property: A subset $S \subset S_X$ is said to be a set of uniformly strongly exposed points if there is a family of functionals $\lbrace f_x \rbrace _{x\in S}$ with $\|f_x\| = f_x(x) = 1$ for every $x \in S$ such that, given $\varepsilon  > 0$ there is a modulus $\omega(\varepsilon) > 0$ satisfying 
\[ \forall z\in B_X, \forall x\in S: \|x-z\|\geq \varepsilon \Longrightarrow 1-\omega(\varepsilon) \geq f_x(z),\]
that is, all elements of $S$ are strongly exposed points with a same  modulus $\omega(\varepsilon)$.  By setting $f_{-x}:=-f_x$ for all $x\in S$, we see that  $S\cup -S$ is a set of uniformly strongly exposed points (by the family $\lbrace f_x \rbrace _{x\in S\cup -S}$) whenever $S\subset S_X$ is a set of uniformly strongly exposed points (by the family $\lbrace f_x \rbrace _{x\in S}$). Thus, we can always assume that a set of uniformly strongly exposed points is symmetric.

Let $C$ be a nonempty bounded subset of $X$. We say that $(C,\|\cdot\|)$ has the uniform separation property (in short, has the $USP$) if for every $\varepsilon>0$ small enough there exists a modulus $\omega(\varepsilon)>0$ and a family $\lbrace f_{x,\varepsilon}\in B_{X^*}: x\in C, \varepsilon>0\rbrace$  such that,
\[ \forall x, y \in C: \|x-y\| \geq \varepsilon \Longrightarrow \langle f_{x,\varepsilon}, x \rangle -  \omega(\varepsilon) \geq  \langle f_{x,\varepsilon}, y \rangle.\]
It is easy to see that a set $(C,\|\cdot\|)$ has the $USP$ if and only if $(\overline{C},\|\cdot\|)$ has the $USP$. The notion of $``USP"$ was introduced  by the author in \cite{Ba}. It has been used recently in \cite{Ba0} in the context of operators attaining their numerical radius. We see  that if $(C,\|\cdot\|)$ is a bounded closed subset with the $USP$, then the subset $\textnormal{se}(C)$ of all strongly exposed points of $C$ is dense in $C$ (see Remark \ref{pro01}). It turns out that this class of sets satisfies a generalized and strong form of Bishop-Phelps property. It also allows to exhibit a new class of Banach spaces satisfying property $(A)$ of Lindenstrauss.  

\begin{definition} Let $X$ be a Banach space.

$(i)$ We say that $X$ has the property $n.u.s.e.$ (norming uniformly strongly exposed points) with a modulus $\omega(\cdot)$, if  $B_X=\overline{\textnormal{co}}(S)$ for some symmetric set $S\subset S_X$ of uniformly strongly exposed points  with the modulus $\omega(\cdot)$. 

$(ii)$ We say that $X$ has the property $n.u.s.$ (norming uniform separation) with a modulus $\omega(\cdot)$, if  $B_X=\overline{\textnormal{co}}(C)$ for some symmetric set $C\subset S_X$ such that $(C,\|\cdot\|)$ has the $USP$ with the modulus $\omega(\cdot)$. 
\end{definition}
For the notion of $\sigma$-porosity, we refer to Definition  ~\ref{prous} in Section ~\ref{S1} and the reference \cite{Za}.
\begin{definition}
We say that a Banach space $X$ has property $(\sigma A)$ if for every Banach space $Y$ and every closed subspace $\mathcal{Z}$ of $\mathcal{L}(X,Y)$ containing the rank-one operators, the subset $NA(X,Y)\cap \mathcal{Z}$ is the complement of a $\sigma$-porous subset of $\mathcal{Z}$. 
\end{definition}
Our main result is  given in Theorem ~\ref{Gprince1} (and its direct consequence Proposition ~\ref{Gprince}. Some techniques used in this article are  inspired by those developed by the author in \cite{Ba0}, but the results are different and new). This result says that if $C\subset X$ has the $USP$ then it satisfies a generalized and a strong form of the Bishop-Phelps property established by Bourgain in \cite{Bj} for sets that are dentable (see also the recent work in \cite{JMZ}). In addition, our result gives a version of the {\it “Bishop-Phelps-Bollobás"} theorem for linear operators and non-linear mappings.  Trivially, the property $(\sigma A)$ implies the property $(A)$ and the property $n.u.s.e.$ implies the property $n.u.s.$ Lindenstrauss proved (see \cite{Lj}) that  the property $n.u.s.e.$ implies property $(A)$.  When we apply our results to a subset $(C,\|\cdot\|)$ having the $USP$ and satisfying $B_X=\overline{\textnormal{co}}(C)$, then we obtain an extension of the Lindenstrauss result, that is: $ \textnormal{ Property } n.u.s. \Longrightarrow \textnormal{ Property } (\sigma A) \hspace{1mm} \textnormal{(see Theorem \ref{cor1})}.$

A space with property $\alpha$ introduced by W. Schachermayer in \cite{Sch} or a uniformly convex space have the property $n.u.s.e$ (in uniformly convex spaces, the unit sphere is a set of uniformly exposed points, see \cite[Lemma 1]{Pbj}). As mentioned above,  the property $n.u.s.e.$  implies in turn the property $n.u.s.$ (introduced in this paper), but the converse is not true in general even in two-dimensional normed spaces (see Remark \ref{recap}). We also prove  in Proposition ~\ref{Gprince001} that the class of Banach spaces with the property $n.u.s.$ is stable under $\ell_1$-sum when the spaces have a “same modulus".  By applying our results to the Lipschitz-free space  (see \cite{GK}, \cite{W}, \cite{AE}) we also give an extension of some results recently obtained  in \cite[Proposition 3.3 \& Corollary 3.8]{CCGMR} (see Section \ref{FS}, Corollary \ref{cor2}). Finally, under the condition that Banach spaces $E_i$ have property $n.u.s.e.$ for $i\in \lbrace 1,...,n\rbrace$, we prove in Corollary \ref{cor03} that the set of bounded multilinear mappings from $\prod_{i=1}^n E_i$ into $W$, {\it attaining their norm} is also the complement of a $\sigma$-porous set.

The paper is organized as follows. In Section \ref{S1}, we will give  the definitions and notions that we will use in the other sections. In Section \ref{MR} we prove our main result Theorem \ref{Gprince1} after proving some lemmas and we will give as a consequence Proposition ~\ref{Gprince}. In Section \ref{U}, we will deal with  property $n.u.s.$ We show that the property $n.u.s.$  is stable under $\ell_1$-sum (Proposition \ref{Gprince001}). It will then be shown that a space with  property $n.u.s.$ has property $(\sigma A)$ (Theorem \ref{cor1}). A result of $\sigma$-porosity of norm non-attaining multilinear mappings will be given in Corollary \ref{cor03} (Section \ref{F4S}). In Section \ref{FS}, we will apply our results to Lipschitz-free spaces  in Corollary \ref{cor2} to obtain some extensions of the recent results obtained in \cite[Proposition 3.3 \& Corollary 3.8]{CCGMR}.

\section{Notions and Definitions} \label{S1}

We first recall some notions and definitions that we are going to use in this paper.
\begin{definition} Let $C$ be a nonempty set and $\gamma : C\times C\to \R^+$. We say that $\gamma$ is a pseudometric if 
	
	$(1)$ $\gamma(x,x)= 0$, for all $x \in C$.
	
	$(2)$ $\gamma(x,y)=\gamma(y,x)$, for all $x \in C$.
	
	$(3)$ $\gamma(x,y)\leq \gamma(x,z)+ \gamma(z,y)$, for all $x,y,z \in C$.
\end{definition}
Unlike a metric space,  one may have $\gamma ( x , y ) = 0$  for distinct values $ x \neq y $. A pseudometric induces an equivalence relation, that converts the pseudometric space into a metric space. This is done by defining $x\sim y$ if and only if $\gamma ( x , y ) = 0$.
Let $\Gamma_\gamma:  C\to  C/\sim~ $ be the canonical surjection mapping and let 
$$d_\gamma(\Gamma_\gamma(x),\Gamma_\gamma(y)):=\gamma(x,y).$$
We set $[C]:=C/\sim~ $. Then, $([C],d_\gamma)$ is a well defined metric space. We say that $(C,\gamma)$ is a complete pseudometric space, if $([C],d_\gamma)$ is a complete metric space. 

\vskip5mm
We  need the following definition of uniform separation property for pseudometrics introduced in \cite{Ba}.
\begin{definition} \label{USP}
	Let $X$ be a Banach space, $C$ be a subset of the dual $X^*$ and $(C,\gamma)$ be a pseudometric space.  We say that  $(C,\gamma)$ has the weak$^*$-uniform separation property in $X^*$, in short $w^*USP$ if  for every $\varepsilon>0$ small enough, there exists $\omega_C(\varepsilon)>0$ and a familly $\mathcal{F}_C:=\lbrace x_{p,\varepsilon} \in B_X: p\in C, \varepsilon >0 \rbrace$ satisfying: for every $p, q \in C$,
	\begin{eqnarray*}
		\gamma(p,q)\geq \varepsilon \Longrightarrow \langle p , x_{p,\varepsilon} \rangle -  \omega_C(\varepsilon) \geq  \langle q , x_{p,\varepsilon} \rangle.
	\end{eqnarray*} 
 If $C$ is a subset of a  Banach space $X$, we say that $(C,\gamma)$ has the $USP$ in $X$ if $(C,\gamma)$ has the $w^*USP$ in $X^{**}$, when $C$ is considered as a subset of the bidual $X^{**}$. \\
The function $\omega_C$ is called, a {\it modulus of uniform separation} of $(C,\gamma)$. \\
\end{definition}
Note that if $\omega_C$ is a modulus of uniform separation of $(C,\gamma)$ then every $0< \delta(\cdot) \leq \omega_C(\cdot)$ is still a modulus of uniform separation.
\begin{example} We reefer to \cite{Ba} for some examples of pseudometric spaces having the uniform separation property. Let $X$ be a uniformly convex Banach space  with a modulus of uniform convexity $\delta(\cdot)$. Then, $(S_{X},\|\cdot\|)$ has the $USP$ in $X$ with a modulus of uniform separation $\delta(\cdot)$ (see \cite[Proposition 2.5]{Ba}).
\end{example}

Recall that a function $f$ has a strong minimum on a metric space $(C,d)$ at some point $p\in C$, if $f$ attains its minimum at $p$ and for any sequence $(p_n)\subset C$ such that $f(p_n)\to f(p)=\inf_C f$, we have that $d(p_n,p)\to 0$. A function $f$ has a strong maximum if $-f$ has a strong minimum. 

To obtain our result in the more general case of pseudometric spaces (useful in Section \ref{MR}), we need to recall the following general definition from \cite{Ba}.
\begin{definition} \label{Gdir} Let $(C,\gamma)$  be a pseudometric space. Let $f: C \to \R \cup \lbrace +\infty \rbrace$ be a proper bounded from below function. We say that $f$ is $\gamma$-strongly minimized on $C$ at $u \in C$  if and only if for every sequence $(q_n)\subset C$ we have 
$$\lim_{n\rightarrow +\infty} f(q_n)=\inf_C f \Longrightarrow \lim_{n\rightarrow +\infty} \gamma(q_n,u)=0.$$
A function $f$ is $\gamma$-strongly maximized on $C$ at $u \in C$ if and only if $-f$ is $\gamma$-strongly minimized on $C$ at $u \in C$.
\end{definition}
In the general case, it may be that in the previous definition we have that $\inf_C f \neq f(u)$. However, the direction $u$ is necessarilly unique up to the relation $\sim$, that is, every other direction $v\in C$ satisfying the above property is such that $\gamma(v,u)=0$ and the converse is also true. Note that if, moreover,  we assume that$f$ is lower semicontinuous with respect to the pseudometric $\gamma$ (that is, for every sequence $(q_n)\subset C$, $\liminf_{n\rightarrow +\infty} f(q_n)\geq f(u)$, whenever $\lim_{n\rightarrow +\infty} \gamma(q_n,u)=0$), then the infimum of $f$ is attained at $u$. In the particular case where $\gamma$  is a metric and $f$ is lower semicontinuous for $\gamma$, the notion of  ``is $\gamma$-strongly minimized on $C$" coincides with the classical notion of {\it strong minimum} mentioned above.
\vskip5mm
Let $(C,\gamma)$ be a pseudometric subset of a Banach space $X$. Suppose that the identity map $i: (C,\gamma) \to (C,\textnormal{weak})$ is continuous (that is, the topology induced by $\gamma$ on $C$ is finer than the weak-topology, this implies in particular that every functional $f\in X^*$ is $\gamma$-continuous on $C$). We say that $x\in C$ is a $\gamma$-strongly exposed point of $C$ if there exists a functional $f\in X^*$ such that $f$ is $\gamma$-strongly maximized on $C$ at $x$. Equivalently, $f(x)=\sup_{y\in C} f(y)$ and for every sequence $(x_n)\subset C$, we have $\gamma(x_n,x)\to 0$ whenever $f(x_n)\to f(x)$. The set of all $\gamma$-strongly exposed points of $C$ will be denoted $\textnormal{se}_\gamma(C)$. When $\gamma$ is the metric given by the norm $\|\cdot\|$, we simply say {\it ``strongly exposed point of $C$"} instead of {\it ``$\gamma$-strongly exposed point of $C$"} and denotes $\textnormal{se}(C)$ instead of $\textnormal{se}_\gamma(C)$.
\vskip5mm

We recall the notion of $\sigma$-porosity. In the following definition,  $\mathring{B}_X(x ; r)$ stands for the open ball in $X$ centered at $x$ and with radius $r > 0$.
\begin{definition}\label{prous}  Let $(X ; d)$ be a metric space and $A$ be a subset of $X$. The set $A$ is
said to be porous in $X$ if there exist $\lambda_0 \in (0; 1]$ and $r_0 > 0$ such that for any $x \in X$
and $r \in (0; r_0]$ there exists $y \in X$ such that $\mathring{B}_X(y; \lambda_0r) \subseteq \mathring{B}_X(x; r) \cap (X \setminus A)$. The set
$A$ is called $\sigma$-porous in $X$ if it can be represented as a countable union of porous sets in $X$.
\end{definition}
 Every $\sigma$-porous set is of first Baire category. Moreover, in $\R^n$, every $\sigma$-porous set is of Lebesque measure zero. However,  there does exist a non-$\sigma$-porous subset of $\R^n$ which is of the first category and of Lebesgue measure zero. For more informations about $\sigma$-porosity, we refer to \cite{Za}. 
\section{Main result} \label{MR}

This section is devoted to the proof of our main result Theorem \ref{Gprince1} below. We start by observing that $\gamma$-strongly exposed points of a complete pseudometric space having the $USP$ are dense.
\begin{remark} \label{pro01} Let $X$ be a Banach space and $(C,\gamma)$ be a norm bounded complete pseudometric space having the $USP$ in $X$ and such that the identity map $i: (C,\gamma)\to (C, \textnormal{weak})$ is continuous. Then, $C=\overline{\textnormal{se}_\gamma(C)}^{\gamma}$ (the closure for the pseudometric $\gamma$). 
\end{remark}
\begin{proof} We apply \cite[Theorem 3.7]{Ba} with the null function in the bidual space $X^{**}$ and by considering $C$ as a subset of $X^{**}$. We get that for every $\varepsilon>0$ and every $x\in C$, there exists a functional $p\in X^*$ such that $p$ is $\gamma$-strongly maximized on $C$ at some point $\bar{x}\in C$ satisfying $\gamma(\bar{x}, x) < \varepsilon$. By the continuity of $i: (C,\gamma)\to (C, \textnormal{weak})$, it follows that $\bar{x}\in \textnormal{se}_\gamma(C)$ and $\gamma(\bar{x}, x) < \varepsilon$.
\end{proof}
In all what follows, we assume that $X,W$ are real Banach spaces and $C\subset X$ will be a symmetric set.  Let $\gamma$ be a pseudometric on $C$ satisfying $\gamma(x,y)=\gamma(-x,-y)$ for all $x,y\in C$. Define 
$$\gamma^s(x, y):=\min(\gamma(x,y), \gamma(x,-y)), \hspace{1mm} \forall x, y \in C.$$

We denote $(\mathcal{C}_b(C,W), \|\cdot\|_{\infty})$ the Banach space of all bounded continuous functions from $(C, \gamma)$ into $W$, equipped with the sup-norm. 
For each $x\in C$ and $w^*\in S_{W^*}$, we denote $w^* \circ \hat{x} \in (\mathcal{C}_b(C,W))^*$ the functional defined by $w^*\circ \hat{x}(T):=w^*(T(x))$ for all $T\in \mathcal{C}_b(C,W)$. We define 
$$\mathcal{K}_C:=\lbrace w^*\circ \hat{x}: x\in C; w^*\in S_{W^*}\rbrace\subset B_{(\mathcal{C}_b(C,W))^*}$$
and for all $w^*\circ \hat{x}, v^*\circ \hat{y}\in \mathcal{K}_C$:
$$\gamma_{\mathcal{P}}(w^*\circ \hat{x}, v^*\circ \hat{y}):=\gamma^s(x,y).$$
The map $\gamma_{\mathcal{P}}$ is generally well defined as we will see in Lemma \ref{ZUSP}. For each fixed $T\in \mathcal{C}_b(C,W)$, we define the map $\hat{T}$ from $\mathcal{K}_C$ into $\R$ by $$\hat{T}: w^*\circ \hat{x} \mapsto w^*(T(x)) \textnormal{  for all } w^*\circ \hat{x}\in \mathcal{K}_C.$$ 
We need the following definition.

\begin{definition} \label{Gdir1} Let $X,W$ be Banach spaces and $C\subset X$ be a symmetric set.  Let $\gamma$ be a pseudometric on $C$ satisfying $\gamma(x,y)=\gamma(-x,-y)$ for all $x,y\in C$. We say that a mapping $T\in \mathcal{C}_b(C,W)$ attains $\gamma^s$-strongly its sup-norm  if there exists $\bar{x} \in C$ such that:

$(a)$ $\|T\|_{\infty}:=\sup_{x\in C}\|T(x)\|=\|T(\bar{x})\|$ and 

$(b)$ for all sequence $(x_n)\subset C$, we have $$\|T(x_n)\| \to \sup_{x\in C}\|T(x)\|  \Longrightarrow \gamma^s(x_n,\bar{x})\to 0.$$
\end{definition}

We now give our main result. It is based on a variational principle given by the author in \cite{Ba}. The following result is complementary to \cite[Theorem ~4.6]{Ba}. The difference is that here we deal with the  more general class of pseudometric spaces having the $USP$  which allows  the space $\mathcal{Z}$ to be chosen as any closed subspace of  $\mathcal{L}(X,W)$ containing the space of rank-one operators. 

 The following result says that a pseudometric space $(C,\gamma)$ having the $USP$ satisfies a generalized and stronger form of the Bishop-Phelps property and also a “quantitative version” of the Bishop-Phelps-Bollobás theorem. On the one hand, the density is replaced by being the complement of a $\sigma$-porous set and on the other hand, the result is applicable for both linear  operator spaces and non-linear mapping spaces. A more classical situation for linear operator spaces when $C$ is equipped with the norm, is given in Proposition ~\ref{Gprince}. Applications to norm attaining bounded linear operators and multilinear mappings are given respectively  in Theorem ~\ref{cor1} and Corollary ~\ref{cor03}  and an application to Lipschitz-free spaces  is given in Corollary \ref{cor2}.

\begin{theorem} \label{Gprince1} Let $X, W$ be Banach spaces and $C\subset X$ be a norm bounded and symmetric set. Let $\gamma$ be a complete pseudometric on $C$ satisfying $\gamma(x,y)=\gamma(-x,-y)$ for all $x,y\in C$. Suppose that the identity mapping $i: (C,\gamma)\to (C,\textnormal{weak})$ is continuous and that $(C,\gamma)$ has the $USP$ in $X$ with a modulus of uniform separation $\omega_C(\cdot)$. Let $W$ be any Banach space and $(\mathcal{Z}, \|\cdot\|_\mathcal{Z})$ be a Banach space included in $\mathcal{C}_b(C,W)$ and satisfy the following properties:

$(i)$   $\|\cdot\|_\mathcal{Z}\geq \|\cdot\|_{\infty}$.

$(ii)$ $\mathcal{Z}$ contains  the set of rank-one operators and is such that $\|x^*(\cdot)e\|_\mathcal{Z}\leq \|x^*\|\|e\|$, for all $x^*\in X^*$ and all $e\in W$.

Then, the following assertions hold.

$(1)$ for every $S\in \mathcal{C}_b(C,W)$, the set 
\begin{eqnarray*}
\mathcal{G}(S):=\lbrace T\in \mathcal{Z}: S+T \textnormal{ attains } \gamma^s\textnormal{-strongly  its sup-norm on } C \rbrace,
\end{eqnarray*}
is the complement of a $\sigma$-porous set of $(\mathcal{Z}, \|\cdot\|_\mathcal{Z})$.  

$(2)$ the following version of the Bishop-Phelps-Bollobás theorem holds: for every $\varepsilon>0$, there exists $\lambda(\varepsilon):=\frac{\varepsilon}{8} \omega_C(\varepsilon/4) >0$ such that for every $T\in Z$, $\|T\|_{\infty}=1$ and every $x\in C$ satisfying $\|T(x)\|>1-\lambda(\varepsilon),$ there exists $S \in Z$, $\|S\|_{\infty}=1$ and $\overline{x}\in C$ such that 

$(a)$ the operator $S$  attains $\gamma^s$-strongly  its sup-norm  at  $\overline{x}$, 

$(b)$ $\gamma^s(\overline{x},x):=\min(\gamma(\bar{x},x), \gamma(\bar{x},-x))<\varepsilon$ and $\|T-S\|_{\infty}<\varepsilon.$
\end{theorem} 
The proof of Theorem \ref{Gprince1} will be given after Proposition \ref{Gprop1},  Lemma ~\ref{ZUSP} and Lemma ~\ref{GgammaH}.
\vskip5mm
Recall that  the Hausdorff distance $d^H(A,B)$ between two nonempty  closed and bounded subsets $A, B$ of some metric space $(E,d)$ is defined as follows.
\begin{eqnarray*}
 d^H ( A , B ) &:=& \max (\sup_{ y\in B}  d(y , A ) , \sup_{x \in A} d ( x , B ) ) \\
&=& \max \lbrace \sup_{y \in B} \inf_{x \in A} d ( x , y ) , \sup_{x\in A} \inf_{y \in B} d ( x , y ) \rbrace.
\end{eqnarray*}
It is well-known that if $(E,d)$ is complete and if we denote $K(E)$ the set of all nonempty compact sets of $E$, then $(K(E), d^H)$ is a complete metric space (see \cite{CK} and the references therein).

\begin{proposition} \label{Gprop1} Let $X$ be a Banach space and $C$ be a nonempty symmetric subset of $X$. Let $\gamma$ be a pseudometric on $C$ satisfying $\gamma(x,y)=\gamma(-x,-y)$ for all $x,y\in C$. Then, the following assertions hold.

$(i)$ $\gamma^s(\pm x, \pm y)=\gamma^s(x, y):=\min(\gamma(x,y), \gamma(x,-y))$ for all $x, y \in C$.

$(ii)$ If $(C,\gamma)$ is a complete pseudometric space, then $(C,\gamma^s)$ is also a complete pseudometric space. 

\end{proposition}

\begin{proof} The part $(i)$ is trivial. To see $(ii)$, since $(C,\gamma)$ is a complete pseudometric space, then by definition $([C],d_\gamma)$ is a complete metric space. Let $(K([C]), d_\gamma^H)$ be the complete metric space of all compact subsets of $([C],d_\gamma)$ equipped with the Hausdorff distance. Consider the subset $P([C])$ of  compact subsets  of the form $\lbrace \Gamma_\gamma (x),\Gamma_\gamma (-x)\rbrace$, $x\in C$, that is,
$$P([C]):=\lbrace \lbrace \Gamma_\gamma (x),\Gamma_\gamma (-x)\rbrace: x\in C\rbrace \subset K([C]).$$
Clearly, the set $P([C])$ is a closed subset of $(K([C]), d^H_\gamma)$. It follows that $(P([C]), d^H_\gamma)$ is a complete metric subspace of $K([C])$. To see that $(C,\gamma^s)$ is a complete pseudometric space it suffices to observe that the mapping
\begin{eqnarray*}
\xi: (C,\gamma^s) &\to & (P([C]), d^H_\gamma)\\
       x &\mapsto& \lbrace \Gamma_\gamma (x),\Gamma_\gamma (-x)\rbrace,
\end{eqnarray*}
is $\gamma^s$-$d^H_\gamma$-isometric. Indeed, we see easily that for all $x, y \in C$
\begin{eqnarray*}
d^H_\gamma(\xi(x), \xi(y)) &=& d^H_\gamma(\lbrace \Gamma_\gamma (x),\Gamma_\gamma (-x)\rbrace, \lbrace \Gamma_\gamma (y),\Gamma_\gamma (-y)\rbrace)\\
&=&\min(d_\gamma(\Gamma_\gamma(x),\Gamma_\gamma(y)),d_\gamma(\Gamma_\gamma(x),\Gamma_\gamma(-y)))\\
&=& \min(\gamma(x,y), \gamma(x,-y))\\
&=& \gamma^s(x,y).
\end{eqnarray*}
Hence, $(C,\gamma^s)$ is a complete pseudometric space. 
\end{proof}

We prove in the following lemma that the map $\gamma_{\mathcal{P}}$ is well defined and that $(\mathcal{K}_C, \gamma_{\mathcal{P}})$ has the $w^*USP$ in $\mathcal{Z}^*$. 

\begin{lemma} \label{ZUSP} Let $X, W$ be  Banach spaces and $C\subset X$ be a norm bounded and symmetric set. Let $\gamma$ be a complete pseudometric on $C$ satisfying $\gamma(x,y)=\gamma(-x,-y)$ for all $x,y\in C$. Suppose that the identity mapping $i: (C,\gamma)\to (C,\textnormal{weak})$ is continuous and that $(C,\gamma)$ has the $USP$ in $X$ with a modulus $\omega_C(\cdot)$. Let $(\mathcal{Z}, \|\cdot\|_\mathcal{Z})$ be a Banach space included in $\mathcal{C}_b(C,W)$ and satisfying the following properties:

$(i)$   $\|\cdot\|_\mathcal{Z}\geq \|\cdot\|_{\infty}$.

$(ii)$ $\mathcal{Z}$ contains  the set of rank-one operators (which are continuous for the topology $\tau_\gamma$ induced by $\gamma$)  and is such that $\|x^*(\cdot)e\|_\mathcal{Z}\leq \|x^*\|\|e\|$, for all $x^*\in X^*$ and all $e\in W$.

 Then, the following assertions hold.

$(a)$  the set $\mathcal{K}_C$ is a subset of $\mathcal{Z}^*$,  

$(b)$ the space $(\mathcal{K}_C, \gamma_{\mathcal{P}})$ is a well defined complete pseudometric space,

$(c)$ the space $(\mathcal{K}_C, \gamma_{\mathcal{P}})$ has the $w^*USP$ in $\mathcal{Z}^*$ with a modulus $\frac{\omega_C(\cdot)}{2}$.
\end{lemma}
\begin{proof} First, notice that $X^*\subset \mathcal{C}_b(C,W)$ this is due to the fact that the topology $\tau_\gamma$ induced by $\gamma$ is finer than the weak-topology. Notice also that for all $x\in C$ and all $w^*\in S_{W^*}$, we have $w^*\circ \hat{x} \in \mathcal{Z}^*$ since $\|\cdot\|_\mathcal{Z}\geq \|\cdot\|_{\infty}$. Thus, $\mathcal{K}_C\subset \mathcal{Z}^*$. By assumption we have: for every $\varepsilon>0$ small enough, there exists $\omega_C(\varepsilon)>0$ and a familly $\mathcal{F}_{C}:=\lbrace x^*_{x,\varepsilon} \in B_{X^*}: x\in C, \varepsilon >0 \rbrace$ satisfying: for every $x,y \in C$, 
\begin{eqnarray*}
\gamma(x,y) \geq \varepsilon \Longrightarrow \langle x^*_{x,\varepsilon}, x \rangle - \omega_C(\varepsilon) \geq  \langle x^*_{x,\varepsilon}, y\rangle.
\end{eqnarray*} 
By the symmetry of $C$, we see that for every $x,y \in C$,
\begin{eqnarray} \label{Geq1}
\gamma^s(x,y) \geq \varepsilon \Longrightarrow \langle x^*_{x,\varepsilon}, x \rangle - \omega_C(\varepsilon) \geq  |\langle x^*_{x,\varepsilon}, y\rangle|.
\end{eqnarray} 
 Recall that $\gamma_{\mathcal{P}}$ is defined for all $w^*\circ \hat{x}, v^*\circ \hat{y}\in \mathcal{K}_C$ by:
$$\gamma_{\mathcal{P}}(w^*\circ \hat{x}, v^*\circ \hat{y}):=\gamma^s(x,y):=\min(\gamma(x,y), \gamma(x,-y)).$$
Let us prove that the map $\gamma_{\mathcal{P}}$ is well defined. To do this, it suffices to show that for $x, y, z, t\in C$ and $w^* , v^*, s^*, q^*\in S_{W^*}$, if $(w^*\circ \hat{x}, s^*\circ \hat{z})= (v^*\circ \hat{y}, q^*\circ \hat{t})$ then $\gamma^s(x,z)= \gamma^s(y,t)$. Indeed, since $\mathcal{Z}$ contains the set of rank-one operators, then from the equality $w^*\circ \hat{x}=v^*\circ \hat{y}$, we get that $x^*(x)w^*(e)=x^*(y)v^*(e)$, for all $x^*\in X^*$ and for all $e\in S_{W}$. This implies that $y=\pm x$.  In a similar way, we see that $t=\pm z$. Hence,  $\gamma^s(x,z)= \gamma^s(y,t)$ and so $\gamma_{\mathcal{P}}$ is well defined. Clearly, $(\mathcal{K}_C, \gamma_{\mathcal{P}})$ is a complete pseudometric space, since $(C, \gamma^s)$ is a complete pseudometric space, this follows from  Proposition \ref{Gprop1}. 

Now, we are going to prove that $(\mathcal{K}_C, \gamma_{\mathcal{P}})$ has the $w^*USP$ in $\mathcal{Z}^*$. 
Set $D:=\sup_{a\in C} \|a\|$ and let $\varepsilon>0$, $x\in C$ and $w^*\in S_{W^*}$. Choose $e_{\varepsilon, w^*}\in S_W$ such that $$w^*(e_{\varepsilon,w^*}) >1- \lambda(\varepsilon),$$ with $\lambda(\varepsilon):=\frac{\omega_C(\varepsilon)}{2D}$. We define $T_{\varepsilon, x, w^*}\in \mathcal{Z}$ by $T_{\varepsilon, x, w^*}:=x^*_{x,\varepsilon}(\cdot)e_{\varepsilon,w^*}$, where $x^*_{x,\varepsilon}\in \mathcal{F}_{C}$. We have $\|T_{\varepsilon, x, w^*}\|_\mathcal{Z}\leq \|x^*_{x,\varepsilon}\|\|e_{\varepsilon,w^*}\|\leq 1$ by the assumption $(ii)$. To see that the space $(\mathcal{K}_C, \gamma_{\mathcal{P}})$ has the $w^*USP$ in $\mathcal{Z}^*$, let $y\in C$ and $v^*\in S_{W^*}$ be such that $\gamma_{\mathcal{P}}(w^*\circ \hat{x}, v^*\circ \hat{y}):=\gamma^s(x,y)\geq \varepsilon$. Then, using the formula $(\ref{Geq1})$, we get:

\begin{eqnarray*}
|v^*\circ \hat{y}(T_{\varepsilon, x, w^*})|&=&|x^*_{x,\varepsilon}(y)v^*(e_{\varepsilon,w^*})|\\
                                                               &\leq& |x^*_{x,\varepsilon}(y)|\\
                                                               &\leq& x^*_{x,\varepsilon} (x) - \omega_C(\varepsilon)\\
&=& x^*_{x,\varepsilon} (x) w^*(e_{\varepsilon,w^*}) + x^*_{x,\varepsilon} (x) (1-w^*(e_{\varepsilon,w^*})) - \omega_C(\varepsilon)\\
&=& w^*\circ \hat{x}(T_{\varepsilon, x, w^*})+ x^*_{x,\varepsilon} (x) (1-w^*(e_{\varepsilon,w^*})) - \omega_C(\varepsilon)\\
&\leq& w^*\circ \hat{x}(T_{\varepsilon, x, w^*}) + D\lambda(\varepsilon) - \omega_C(\varepsilon)\\
                                                               &=& w^*\circ \hat{x}(T_{\varepsilon, x, w^*}) - \frac{\omega_C(\varepsilon)}{2}.
\end{eqnarray*}
Hence, by Definition \ref{USP}, $(\mathcal{K}_C, \gamma_{\mathcal{P}})$ is a complete pseudometric space having the $w^*USP$ in $\mathcal{Z}^*$ with a modulus $\frac{\omega_C(\varepsilon)}{2}$.
\end{proof}

\begin{lemma} \label{GgammaH} Let $X, W$ be Banach spaces and $C$ be a symmetric subset of $X$. Let $\gamma$ be a pseudometric on $C$ satisfying $\gamma(x,y)=\gamma(-x,-y)$ for all $x,y\in C$. Let $T\in \mathcal{C}_b(C,W)$. Then, $T$ attains $\gamma^s$-strongly its sup-norm  at some point $\bar{x}\in C$ (Definition \ref{Gdir1}) if and only if the function $\hat{T}: w^*\circ \hat{x}\mapsto w^*(T(x))$ is $\gamma_{\mathcal{P}}$-strongly maximized on $\mathcal{K}_C$ at some point $w^*_0\circ \hat{x}_0 \in\mathcal{K}_C$ (Definition \ref{Gdir}). In this case, we have $\bar{x}=\pm x_0$.
\end{lemma}
\begin{proof} Suppose that $\hat{T}$  is $\gamma_{\mathcal{P}}$-strongly maximized on $ \mathcal{K}_C $ at somme point $w^*_0\circ \hat{x}_0 \in \mathcal{K}_C$ . Let $(x_n)\subset C$ be such that $\|T(x_n)\|\to \sup_{x\in C}\|T(x)\|$. By the Hahn-Banach theorem, there exists a sequence $(w_n^*)\subset S_{W^*}$  such that $$\|T(x_n)\|=w_n^*(T(x_n)) \textnormal{ for all } n\in \N$$ and $$\sup_{x\in C}\|T(x)\|=\sup_{x\in C, w^*\in S_{W^*}}w^*(T(x)) =\sup_{w^*\circ \hat{x}\in \mathcal{K}_C}w^*(T(x)).$$ 
In consequence, we have
 $$w_n^*(T(x_n)) \to \sup_{w^*\circ \hat{x}\in \mathcal{K}_C}w^*(T(x)).$$
Since, $T$  is  $\gamma_{\mathcal{P}}$-strongly maximized on  $\mathcal{K}_C$ at $w^*_0\circ \hat{x_0} \in \mathcal{K}_C$, it follows  from Definition \ref{Gdir} that $\gamma_{\mathcal{P}}(w^*_n\circ \hat{x}_n, w^*_0\circ \hat{x}_0)\to 0$, that is, $\gamma^s(x_n,x)\to 0.$ 
It suffices now to see that $\sup_{x\in C}\|T(x)\|=\max(\|T(x_0)\|, \|T(-x_0)\|)$ after which we will take $\bar{x}=\pm x_0$. Indeed, since $\gamma^s(x_n,x)\to 0$,  there exists a subsequence $(x_{n_k})$ such that $\gamma(x_{n_k},x_0)\to 0$ or $\gamma(x_{n_k},-x_0) \to 0$. It follows by the continuity of $T$ that $$\|T(x_{n_k})\|\to \|T(x_0)\|\leq \max(\|T(x_0)\|, \|T(-x_0)\|)$$ or $$\|T(x_{n_k})\|\to \|T(-x_0)\|\leq \max(\|T(x_0)\|, \|T(-x_0)\|).$$ Since, $\|T(x_n)\|\to \sup_{x\in C}\|T(x)\|$, we get $$\sup_{x\in C}\|T(x)\|=\max(\|T(x_0)\|, \|T(-x_0)\|).$$  
Thus,  $T$ attains $\gamma^s$-strongly its sup-norm at $\bar{x}=\pm x_0\in C$. To see the converse, suppose that $T$ attains $\gamma^s$-strongly its sup-norm at $\bar{x}\in C$.  Let $(w^*_n\circ \hat{x}_n)\subset \mathcal{K}_C$ be a sequence  such that $w_n^*(T(x_n)) \to \sup_{w^*\circ \hat{x}\in \mathcal{K}_C}w^*(T(x))=\sup_{x\in C}\|T(x)\|$. Since $$w_n^*(T(x_n))\leq \|T(x_n)\|\leq \sup_{x\in C}\|T(x)\|,$$ 
then,  $$\|T(x_n)\| \to \sup_{x\in C}\|T(x)\|.$$ 
Using the fact that $T$ attains $\gamma^s$-strongly its sup-norm at $\bar{x} \in C$, we get that $\gamma^s(x_n,\bar{x})\to 0,$ which implies by the definition that $\gamma_{\mathcal{P}}(w^*_n\circ \hat{x}_n, w^*_0\circ \hat{\bar{x}})\to 0$, for any fixed $w^*_0\in S_{W^*}$. Hence, $\hat{T}$ is $\gamma_{\mathcal{P}}$-strongly maximized on $\mathcal{K}_C$.
\end{proof}

\begin{proof}[Proof of Theorem \ref{Gprince1}] Using Lemma \ref{ZUSP}, we know that $(\mathcal{K}_C, \gamma_{\mathcal{P}})$ is a complete pseudometric space having the $w^*USP$ in $\mathcal{Z}^*$. Let us fix $S\in \mathcal{C}_b(C,W)$. Thus, by applying \cite[Theorem 3.3]{Ba} to the function 
\begin{eqnarray*}
\hat{S}: (\mathcal{K}_C, \gamma_{\mathcal{P}}) &\to& \R\\
w^*\circ \hat{x}&\mapsto& w^*(S(x)), 
\end{eqnarray*}
we obtain that the set
\begin{eqnarray*}
\mathcal{H}(S):=\lbrace T\in \mathcal{Z}: \hat{S}+\hat{T} \textnormal{ is } \gamma_{\mathcal{P}}\textnormal{-strongly maximized on } \mathcal{K}_C\rbrace,
\end{eqnarray*}
is the complement of a $\sigma$-porous set on $(\mathcal{Z}, \|\cdot\|_\mathcal{Z})$. Using Lemma \ref{GgammaH}, we see that $\mathcal{G}(S)=\mathcal{H}(S)$. This concludes the first part of the theorem. Now, to prove the version of the Bishop-Phelps-Bollobás theorem, we apply \cite[Theorem 3.7]{Ba} with the space $(\mathcal{K}_C, \gamma_{\mathcal{P}})$. Indeed,  let $\varepsilon>0$, $\lambda(\varepsilon)=\frac{\varepsilon}{4} \omega_{\mathcal{K}_C}(\varepsilon/4)=\frac{\varepsilon}{8} \omega_C(\varepsilon/4)>0$, (where $\omega_{\mathcal{K}_C}$ is the modulus of uniform $w^*\mathcal{USP}$ of $(\mathcal{K}_C,\gamma_\mathcal{P})$ in $\mathcal{Z}^*$, see Lemma \ref{ZUSP}).  Let, $T\in Z$, $\|T\|_{\infty}=1$ and $x\in C$ such that 
\begin{eqnarray*}
\|T(x)\|>1-\frac{\varepsilon}{4} \omega_{\mathcal{K}_C}(\varepsilon/4)&=&\|T\|_{\infty}-\frac{\varepsilon}{4} \omega_{\mathcal{K}_C}(\varepsilon/4). 
\end{eqnarray*}
We have that $$1=\|T\|_{\infty}=\sup_{w^*\circ \hat{z}\in \mathcal{K}_C} w^*(T(z)):=\sup_{w^*\circ \hat{z}\in \mathcal{K}_C} \hat{T}(w^*\circ \hat{z}).$$ Moreover, there exists by the Hanh-Banach theorem an $w^*_x\in S_{W^*}$ such that 
$$ \hat{T}(w^*_x\circ \hat{x}):=w^*_x (T(x))= \|T(x)\|.$$ Thus, the above inequality can be written as follows:
\begin{eqnarray*}
\hat{T}(w^*_x\circ \hat{x}) > \sup_{w^*\circ \hat{z}\in \mathcal{K}_C}\hat{T}(w^*\circ \hat{z}) -\frac{\varepsilon}{4} \omega_{\mathcal{K}_C}(\varepsilon/4). 
\end{eqnarray*}
 We apply \cite[Theorem 3.7]{Ba}, with the function $\hat{T}$ (by changing the term ``minimized" by ``maximized") and the set $\mathcal{K}_C$ to obtain some $R\in \mathcal{Z}$ and a point $w^*_0 \circ \hat{x}_0\in \mathcal{K}_C$ such that

 $(\bullet)$ $\gamma_\mathcal{P}(w^*_0 \circ \hat{x}_0, w^*_x\circ \hat{x}):=\gamma^s(x_0,x)<\frac{\varepsilon}{4},$ 

$(\bullet)$ $\|R\|_{\infty}\leq \|R\|_{\mathcal{Z}}<\frac{\varepsilon}{2}$, 

$(\bullet)$ $\hat{T}-\hat{R}$  is $\gamma_\mathcal{P}$-strongly maximized on $\mathcal{K}_C$ at  the point $w^*_0 \circ \hat{x}_0$. 

This leads, thanks to Lemma \ref{GgammaH}, that $T-R$ attains $\gamma^s$-strongly its sup-norm at $\overline{x}=\pm x_0$. From the definition of $\gamma^s$,  we have $\gamma^s(\overline{x},x)=\gamma^s(x_0,x)<\frac{\varepsilon}{4}$. Set $S:=\frac{T-R}{\|T-R\|_{\infty}}$, then  $S$ attains $\gamma^s$-strongly its sup-norm at $\overline{x}$,  $\|S\|_{\infty}=1$  and using the triangular inequality, 
\begin{eqnarray*}
\|T-S\|_{\infty}=\|T- \frac{T-R}{\|T-R\|_{\infty}}\|_{\infty} &=& \|R+ (T-R-\frac{T-R}{\|T-R\|_{\infty}})\|_{\infty}\\
                          &\leq&  \|R\|_{\infty}+ \left \|\frac{1-\|T-R\|_{\infty}}{\|T-R\|_{\infty}}(T-R) \right \|_{\infty}\\
&=&  \|R\|_{\infty}+ |1-\|T-R\|_{\infty} |\\
&=& \|R\|_{\infty}+ |\|T\|_{\infty}-\|T-R\|_{\infty}|\\
&\leq& 2\|R\|_{\infty}\\
                          &<& \varepsilon. 
\end{eqnarray*}
This concludes the proof.
\end{proof}

By $\overline{\mathcal{F}(X,W)}$ we denote the closure of $\mathcal{F}(X,W)$ in $\mathcal{L}(X,W)$.

\begin{corollary} \label{cor0} Under the hypothesis of Theorem \ref{Gprince1}, let $(S_n)\subseteq \mathcal{C}_b(C,W)$ be a sequence. Then, for every $\varepsilon >0$, there exists a compact operator $T_\varepsilon \in \overline{\mathcal{F}(X,W)}$  such that $\|T_\varepsilon\|_{\infty}:=\sup_{x\in C}\|T(x)\|<\varepsilon$ and $S_n+T_\varepsilon$  attains  $\gamma^s$-strongly its sup-norm on $ C$ for every $n\in \N$.
\end{corollary}
\begin{proof}  We apply Theorem \ref{Gprince1}  to the closed subspace $\mathcal{Z}:=\overline{\mathcal{F}(X,Y)}$ of $\mathcal{L}(X,W)$. Thus, for each $n\in \N$, the set 
$$\mathcal{G}(S_n):=\lbrace T\in \overline{\mathcal{F}(X,Y)}: S_n+ T \textnormal{  attains } \gamma^s\textnormal{-strongly its sup-norm on } C\rbrace,$$
is the complement of a $\sigma$-porous subset of $\mathcal{Z}$.  Thus,  $\cap_{n\in \N} \mathcal{G}(S_n)$ is also the complement of a $\sigma$-porous subset of $\mathcal{Z}$, in particular it is a dense subset. Hence, for every $\varepsilon >0$, there exists $T_\varepsilon \in \cap_{n\in \N} \mathcal{G}(S_n)\subseteq \overline{\mathcal{F}(X,Y)}$ such that $\|T_\varepsilon\|_{\infty}<\varepsilon$ and $S_n+T_\varepsilon$  attains  $\gamma^s$-strongly its sup-norm on $C$ for every $n\in \N$.
\end{proof}

\paragraph{\bf A strong Bishop-Phelps property.} We obtain immediately the following consequence from Theorem \ref{Gprince1} in the case where $C$ is equipped with the norm $\|\cdot\|$ and  $\mathcal{Z}$ is a closed subset of $\mathcal{L}(X,Y)$ containing the set of rank-one operators. Thus, the closed convex hull of a bounded closed symmetric set with the $USP$ has a {\it ``strong Bishop-Phelps"} property. With $X$ a uniformly convex space and $C=S_X$, the following proposition recover and extend the result in \cite[Theorem 3.1]{KL}. By the pseudometric $\|\cdot\|^s$ on $C$, we denote the pseudometric $\gamma^s$ where $\gamma(x,y)=\|x-y\|$ for all $x,y\in C$.

\begin{proposition} \label{Gprince} Let $X$ be a Banach space and $C\subset X$ be a nonempty bounded, closed and symmetric subset such that $(C,\|\cdot\|)$ has the $USP$ with a modulus $\omega_C(\cdot)$. Then, for every Banach space $Y$ and for every closed subspace $\mathcal{Z}$ of $\mathcal{L}(X,Y)$ containing the set of rank-one operators, the following assertions hold.

$(i)$ For  every $S\in \mathcal{C}_b(C, Y)$, the set 
\begin{eqnarray*}
\mathcal{G}(S) &:=&\lbrace T\in \mathcal{Z}: S+T \textnormal{ attains } \|\cdot\|^s\textnormal{-strongly its sup-norm on } C\rbrace,
\end{eqnarray*}
is the complement of a $\sigma$-porous set of $\mathcal{Z}$ (in particular, with $S=0$ we have that $C$ and also $\overline{\textnormal{co}}(C)$ have  the Bishop-Phelps property). 

$(ii)$ The following version of the Bishop-Phelps-Bollobás theorem holds: for every $\varepsilon>0$, there exists $\lambda(\varepsilon):=\frac{\varepsilon}{8} \omega_C(\varepsilon/4) >0$ such that for every $T\in \mathcal{Z}$, $\|T\|_{\infty}=1$ and every $x\in C$ satisfying $\|T(x)\|>1-\lambda(\varepsilon),$ there exists $S \in \mathcal{Z}$ and $\overline{x}\in C$ such that: 

$(a)$ $\|S\|_{\infty}=\|S(\bar{x})\|=1$ and $S$ attains $\|\cdot\|^s$-strongly its sup-norm on $C$ at $\bar{x}$.

$(b)$ $\|\bar{x}-x\|<\varepsilon \textnormal{ and } \|T-S\|_{\infty}<\varepsilon.$
\end{proposition}

\begin{remark} The above proposition implies that the closed unit ball $B_X$ has the ``strong" Bishop-Phelps property and also the Bishop-Phelps-Bollobás property whenever $B_X=\overline{\textnormal{co}}(C)$ for some bounded closed symmetric set $C$ having the $USP$. In particular, the closed unit ball of a uniformly convex Banach space satisfies these properties since $(S_X,\|\cdot\|)$ has the $USP$ (\cite[Proposition 2.5]{Ba}) or more generally if $B_X=\overline{\textnormal{co}}(S)$ for some set $S\subset S_X$ of uniformly strongly exposed points (ex. $X=\ell_1$, see \cite{Sch}). This situation will be further explored in Section \ref{U} and Section \ref{Sd}, in order to provide some extensions to the subject of norm attaining operators (Theorem ~\ref{cor1}, Corollary ~\ref{cor03} and Corollary ~\ref{cor2}).
\end{remark}

\section{Banach spaces with property $n.u.s.$} \label{U}

Recall that a point $x\in K\subset X$ is a strongly exposed point if there exists a functional $f\in B_{X^*}$ such that $f$ has a maximum on $K$ at $x$ and for every sequence $(x_n)\subset K$, we have $\|x_n-x\|\to 0$ whenever $f(x_n)\to f(x)$. The set of all strongly exposed points of a set $K$ is denoted $\textnormal{se}(K)$.  It is easy to see that if $S$ is uniformly strongly exposed set (see the definition in the introduction), then $\overline{S}$ is also a uniformly strongly exposed set (see \cite[Fact 4.1]{CGMR}). Thus,  it is quite easy to give an example of bidimensional space not satisfying the property $n.u.s.e.$ 
\begin{example} \label{Nfinite}
Let $A:=B((0,1), 1)$ be the closed Euclidean norm in $\R^2$ centered at $(0,1)$ with radius $1$. Let $\mathcal{N}$ be the norm whose closed unit ball is $$\B_{\mathcal{N}}:=\textnormal{co}(A\cup - A).$$  We see easily that $\textnormal{se}(B_{\mathcal{N}})=\lbrace \pm (x,y)\in \R^2: x\in ]-1,1[, y\in]1,2]\rbrace$. But the point $(1,1)\in \overline{\textnormal{se}(B_{\mathcal{N}})}$ is not an exposed point of $\B_{\mathcal{N}}$. Hence, $\textnormal{se}(B_{\mathcal{N}})$ cannot be a set of  uniformly strongly exposed points. Consequently, $(\R^2,\mathcal{N})$ does not have  the property $n.u.s.e.$ 
\end{example}

As stated in the introduction, it is easy to see from the definitions that the property $n.u.s.e$ implies the property $n.u.s.$ but the converse is not true even in dimension two.  Indeed, we prove in  Proposition \ref{finite} below, that every two-dimensional normed space has  the property $n.u.s$, while we know from  Example \ref{Nfinite} that this is not the case for the property $n.u.s.e$. We also prove in Proposition \ref{Gprince001} that  the property $n.u.s.$ is stable under $\ell_1$-sum, when the spaces have a same modulus of uniform separation.  

Let $X$ be a real Banach space, $A$ be a nonempty subset of $X$, $f\in X^*$ such that $\|f\|=1$ and $\alpha >0$. We define,
$$S(f,\alpha, A)=\left\{x\in A: f(x) > \sup_{a\in A} f(a) -\alpha \right\}.$$
Such a set is called a slice of $A$ and the fact that $\|f\|=1$ and $\alpha >0$
will always be understood when we refer to a “slice of $A$.” 

A point $x \in A \subseteq X$ is called a denting point
of $A$ if for each $\varepsilon >0$ there is a slice $S(f, \alpha, A)$ of diameter less than $\varepsilon$
which contains $x$ and is determined by a functional $f\in X^*$. A strongly exposed point is a denting point, but the converse is not true in general. We denote $\textnormal{dent}(A)$, the set of all denting points of $A$ and $\textnormal{ext}(A)$ the set of all extreme points of $A$. It is well-known that in finite dimensional normed space $X$, we have $\textnormal{ext}(B_X)=\textnormal{dent}(B_X)$ (see \cite{LLT}).
\vskip5mm

\begin{proposition} \label{finite} Let $K$ be a nonempty convex compact set of a Banach space $X$. Suppose that $\textnormal{dent}(K)$ is closed (compact), then  $(\textnormal{dent}(K),\|\cdot\|)$ has the $\textnormal{u.s.p}$ in $X$. In consequence, if $X$ is a finite dimensional normed space such that $\textnormal{ext}(B_X)$ is closed, then $X$ has the property $n.u.s$ (in particular, a two-dimensional space always has the property $n.u.s$).
\end{proposition}
\begin{proof} We know that $\textnormal{dent}(K)\neq \emptyset$ (see for instance \cite[Proposition 3.]{Ba1}). By the definition of denting points, for every $x\in \textnormal{dent}(K)$ and every $\varepsilon >0$, there exists a slice $S(p_{x, \varepsilon}, \alpha(x,\varepsilon),A)$ of diameter less than $\frac{\varepsilon}{2}$ which contains $x$. Since $x\in S(p_{x, \varepsilon}, \alpha(x,\varepsilon),A)$, we have $p_{x, \varepsilon}(x)> \sup_{a \in K} p_{x, \varepsilon}(a) - \alpha(x,\varepsilon)$. Let $y\in K$ be such that $\|y-x\|\geq \varepsilon$. Then, $y \not \in S(p_{x, \varepsilon}, \alpha(x,\varepsilon),A)$, so $p_{x, \varepsilon}(y)\leq \sup_{a \in K} p_{x, \varepsilon}(a) - \alpha(x, \varepsilon)$.  It follows that, $p_{x, \varepsilon}(x)> \sup_{a \in K} p_{x, \varepsilon}(a) - \alpha(x,\varepsilon)\geq p_{x, \varepsilon}(y)$. In other words, for every $x\in \textnormal{dent}(K)$ and every $\varepsilon >0$, there are $p_{x, \varepsilon}\in S_{X^*}$ and $\delta(x, \varepsilon)>0$ such that for all $y\in K$, 
	\begin{eqnarray} \label{Q1}
		\|x-y\| \geq \varepsilon \Longrightarrow p_{x, \varepsilon}(x) -\delta(x, \varepsilon) \geq p_{x, \varepsilon}(y).
	\end{eqnarray}
	Set $\Delta(x,\varepsilon):=\frac{1}{2}\min(\delta(x, \varepsilon), \varepsilon)>0$ for all $\varepsilon >0$ and for all $x\in \textnormal{dent}(K)$. Clearly, we have for every $\varepsilon>0$, 
	$$\textnormal{dent}(K)\subset \bigcup \left\lbrace \mathring{B}(x, \Delta(x,\varepsilon)): x\in \textnormal{dent}(K) \right\rbrace,$$
	where, $\mathring{B}(x,r)$ denotes the open ball centered at $x$ with radius $r>0$. Since $\textnormal{dent}(K)$ is a compact set, there exists an integer number $n(\varepsilon)\in \N$ and points  $a_1,...,a_{n(\varepsilon)} \in \textnormal{dent}(K)$, such that $$\textnormal{dent}(K)\subset \cup_{i=1}^{n(\varepsilon)} \mathring{B}(a_i, \Delta(a_i,\varepsilon)).$$  
	For every  $\varepsilon >0$ and  every $x\in \textnormal{dent}(K)$, choose an index $1\leq i(x,\varepsilon)\leq n(\varepsilon)$ such that $$\|x-a_{i(x,\varepsilon)}\|< \Delta(a_{i(x,\varepsilon)},\varepsilon).$$ 
	Set
	$$F_{\textnormal{dent}(K)}:=\lbrace p_{a_{i(x,\varepsilon)}, \varepsilon}: x\in \textnormal{dent}(K), \varepsilon >0\rbrace.$$
	Let us define $\omega(\varepsilon):=\frac{1}{2}\min_{1\leq i\leq n(\frac{\varepsilon}{2})} \lbrace \delta(a_i, \frac{\varepsilon}{2})\rbrace >0$ for every $\varepsilon >0$.  We are going to prove that $(\textnormal{dent}(K), \|\cdot\|)$ has the $\textnormal{u.s.p}$ in $X$ with the modulus $\omega$ and the familly $F_{\textnormal{dent}(K)}$. Indeed, let $x\in \textnormal{dent}(K)$ and  $y\in K$ such that $\|x-y\|\geq 2\varepsilon$. Then, 
	$$\|a_{i(x,\varepsilon)} -y\|\geq \|x-y\|- \|a_{i(x,\varepsilon)}-x\|\geq 2\varepsilon -  \Delta(a_{i(x,\varepsilon)}, \varepsilon) \geq \varepsilon.$$ This implies by $(\ref{Q1})$ that 
	\begin{eqnarray} \label{Q3}
		p_{a_{i(x,\varepsilon)}, \varepsilon}(a_{i(x,\varepsilon)})  -\delta(a_{i(x,\varepsilon)}, \varepsilon) \geq p_{a_{i(x,\varepsilon)}, \varepsilon}(y).
	\end{eqnarray}
	Now, since $\|x-a_{i(x,\varepsilon)}\|< \Delta(a_{i(x,\varepsilon)},\varepsilon)\leq \frac{\delta(a_{i(x,\varepsilon)}, \varepsilon)}{2}$ and $p_{a_{i(x,\varepsilon)},\varepsilon}\in S_{X^*}$, we obtain  that
	\begin{eqnarray} \label{Q4}
		p_{a_{i(x,\varepsilon)},\varepsilon}(a_{i(x,\varepsilon)})\leq p_{a_{i(x,\varepsilon)},\varepsilon}(x) + \frac{\delta(a_{i(x,\varepsilon)}, \varepsilon)}{2}.
	\end{eqnarray}
	Thus, by combining $(\ref{Q3})$ and $(\ref{Q4})$, we get  
	\begin{eqnarray*} 
		p_{a_{i(x,\varepsilon)},\varepsilon}(x) -\frac{\delta(a_{i(x,\varepsilon)}, \varepsilon)}{2}\geq p_{a_{i(x,\varepsilon)},\varepsilon}(y).
	\end{eqnarray*}
	It follows from the definition of $\omega(\varepsilon)$ (since $1\leq i(x,\varepsilon) \leq n(\varepsilon)$) that
	\begin{eqnarray*} 
		p_{a_{i(x,\varepsilon)},\varepsilon}(x) -\omega(2\varepsilon) \geq p_{a_{i(x,\varepsilon)},\varepsilon}(y),
	\end{eqnarray*}
	where $\omega$ is a modulus not depending on $x$ and $y$. Thus, $(\textnormal{dent}(K),\|\cdot\|)$ has the $\textnormal{u.s.p}$ in $X$. The last part of the proposition is clear.
\end{proof}

We prove now that property $n.u.s.$ is stable under $\ell_1$-sum.

\begin{proposition} \label{Gprince001} For all $\lambda \in \Lambda $ (an index set), let $(X_\lambda, |\cdot|_\lambda)$ be a  Banach spaces having the property $n.u.s.$  with a modulus $\omega_\lambda(\cdot)$. Assume that $\omega(\varepsilon):=\inf_{\lambda \in \Lambda} \omega_\lambda (\varepsilon)>0$ for every $\varepsilon >0$. Then, the Banach space $X=\left (\underset{\lambda\in \Lambda}{\bigoplus} X_\lambda \right)_{\ell_1}$ has the property $n.u.s.$
\end{proposition} 

\begin{proof} For each  $\lambda \in \Lambda$,  let $C_\Lambda$ be a set having the $USP$ in $X_\lambda$ with the modulus $\omega_\lambda$ and satisfying $B_{X_\lambda}=\overline{\textnormal{co}}^{|\cdot|_\lambda}(C_\lambda)$. We are going to prove that the set $C:=\cup_{\lambda \in \Lambda} C_\lambda$ has the the $USP$ in $X$ with the modulus $\omega (\cdot) =\inf_{\lambda \in \Lambda} \omega_\lambda (\cdot)$ and satisfies  $B_X=\overline{\textnormal{co}}^{\|\cdot\|}(C)$. 
Notice that $C_\lambda \cap C_\beta=\emptyset$, whenever $\lambda \neq \beta$ (since $C_\lambda\subset S_{X_\lambda}$ for all $\lambda \in \Lambda$ and $X_\lambda \cap X_\beta=\lbrace 0\rbrace$ for all $\lambda, \beta \in \Lambda$ , $\lambda \neq \beta$).

Let $x\in \cup_{\lambda \in \Lambda} C_\lambda $ be a fixed point. There exists a unique index $i_{x}\in \Lambda $ such that $x=x_{i_x} \in C_{i_x}$. Since $(C_{i_x}, |\cdot|_{i_x})$ has the $USP$ in $X_{i_x}$ with the modulus $\omega_{i_x}(\cdot)$, there exists a familly $\lbrace f_{i_x, a,\varepsilon} \in B_{X^*_{i_x}}: a\in C_{i_x}, \varepsilon >0 \rbrace$ satisfying: for every $a,b \in C_{i_x}$, 
\begin{eqnarray} \label{RT}
|a-b|_{i_x} \geq \varepsilon \Longrightarrow \langle f_{i_x,a,\varepsilon}, a \rangle - \omega_{C_{i_x}}(\varepsilon) \geq  \langle f_{i_x, a,\varepsilon}, b\rangle.
\end{eqnarray} 
In particular, we have for all $a\in C_{i_x}$ (with $b=-a$, since $C_{i_x}$ is symmetric),
\begin{eqnarray} \label{RT1}
 \langle f_{i_x, a,\varepsilon}, a \rangle \geq  \frac{\omega_{C_{i_x}}(\varepsilon)}{2}\geq \frac{\omega(\varepsilon)}{2}.
\end{eqnarray}
Let $\pi_\beta: X \to X_\beta$ be the canonical projection mapping defined by $\pi_\beta(z)=z_\beta$ for all $z=\underset{\lambda\in \Lambda}{\oplus} z_\lambda \in X$. We define $l^*_{x,\varepsilon}\in B_{X^*}$ by $$l^*_{x,\varepsilon}(z)=f_{i_x, x_{i_x},\varepsilon}(\pi_{i_x}( z)), \hspace{2mm} \forall z\in X.$$ 
Now, we prove that the the familly $\lbrace l^*_{x,\varepsilon} \in B_{X^*}: x\in C, \varepsilon >0\rbrace$ satisfies the following property: $\forall x, y \in C$
\begin{eqnarray} \label{RT2}
\|x-y\|\geq \varepsilon \Longrightarrow \langle l^*_{x,\varepsilon}, x \rangle - \frac{\omega (\varepsilon)}{2} \geq  \langle l^*_{x,\varepsilon}, y\rangle.
\end{eqnarray} 
Indeed, there are  unique indexes $i_x, i_y \in \Lambda$ such that $x=x_{i_x}\in C_{i_x}$ and $y=y_{i_x} \in C_{i_y}$. If $i_x=i_y$, then $x_{i_x}, y_{i_x}\in C_{i_x}$,  $\|x-y\|=|x_{i_x}-y_{i_x}|_{i_x}$, $\langle l^*_{x,\varepsilon}, x \rangle=f_{i_x, x_{i_x},\varepsilon}(x_{i_x})$ and $\langle l^*_{x,\varepsilon}, y \rangle=f_{i_x,x_{i_x},\varepsilon}(y_{i_x})$. Thus, $(\ref{RT2})$ is satisfied thanks to $(\ref{RT})$. If $i_x\neq i_y$, then necessarily $y_{i_x}=\pi_{i_x}(y)=0$ and so $(\ref{RT2})$ is also satisfied thanks to the remark in $(\ref{RT1})$. Hence, in all cases $(\ref{RT2})$ is satisfied which means that $(C,\|\cdot\|)$ has the the $USP$ in $X$ with the modulus $\frac{\omega(\varepsilon)}{2}$. Now, we see easily that $B_X=\overline{\textnormal{co}}^{\|\cdot\|}(C)$. Indeed, $B_X=\overline{\textnormal{co}}^{\|\cdot\|}(\cup_{\lambda \in \Lambda} B_{X_\lambda})=\overline{\textnormal{co}}^{\|\cdot\|}(\cup_{\lambda \in \Lambda} \overline{\textnormal{co}}^{|\cdot|_\lambda}(C_\lambda))=\overline{\textnormal{co}}^{\|\cdot\|}(C)$. 

\end{proof}

\begin{example} \label{Exnus} Let $\Lambda$ be an index set.

$(i)$  With $X_\lambda =H$ for all $\lambda \in \Lambda$ where $H$ is a uniformly convex Banach space, the space $X=\left (\underset{\lambda\in \Lambda}{\bigoplus} X_\lambda \right)_{\ell_1}$, which is not uniformly convex, has the property $n.u.s.$

$(ii)$ For every norm $\|\cdot\|$ on $\R^2$, the space $\left (\underset{\lambda\in \Lambda}{\bigoplus} \R^2\right)_{\ell_1}$ has the property $n.u.s.$ (see Proposition \ref{finite} and Proposition \ref{Gprince001}), but does not have the property $n.u.s.e.$ in general (see Example \ref{Nfinite}).
\end{example}
\noindent Recall  the property quasi-$\alpha$ introduced by Choi-Song  \cite{CS} which, in spite of being weaker than property $\alpha$, still implies property $(A)$.  A Banach space $X$ has property quasi-$\alpha$ if there exist $\lbrace x_\lambda: \lambda \in \Lambda \rbrace$, $\lbrace x^*_\lambda: \lambda \in \Lambda \rbrace$, subsets of $X$ and $X^*$  respectively and a map $\rho: \Lambda \to \R$ such that 

$(i)$ $\|x_\lambda\|=\|x^*_\lambda\|=x^*_\lambda (x_\lambda) =1$ for all $\lambda \in \Lambda$.

$(ii)$  $|x^*_\lambda (x_\mu)|\leq \rho(\mu)<1,$ for all $\lambda \neq \mu$,

$(iii)$ For every $e \in \textnormal{ext} (B_{X^{**}})$, there exists a subset $\Lambda_e \subset \Lambda$ such that either $e$ or $-e$ belong to $\overline{\Lambda_e}^{w^*}$ and $r_e := \sup \lbrace \rho(\mu): x_\mu \in \Lambda_e \rbrace < 1$.

It follows that if $\lbrace x_\lambda: \lambda \in \Lambda \rbrace$ witnesses that $X$ has property quasi-$\alpha$, then
$$B_X=\overline{\textnormal{co}}(\lbrace x_\lambda: \lambda \in \Lambda \rbrace).$$
A Banach space $X$ has the property $\alpha$ (W. Schachermayer \cite{Sch}) if the map $\rho(\cdot)$ in the above definition is equal to a constant $r\in [0,1[$. Note that property quasi-$\alpha$ is stable under finite $\ell_1$-sum (see \cite{CS}). However, like property $n.u.s.e.$, property quasi-$\alpha$ may not be satisfied even in two-dimensional space.
\begin{proposition} \label{quasi1} The Euclidean space $X=(\R^n,\|\cdot\|_2)$ does not satisfy property quasi-$\alpha$.
\end{proposition}
\begin{proof} Suppose that $(\R^n,\|\cdot\|_2)$ has property quasi-$\alpha$. By the Milman's theorem (the partial converse of Krein-Milman theorem, see \cite{M1}) and since every point of the sphere is an extreme point, we have $S_X=\overline{\textnormal{ext}(B_X)}\subset\overline{\lbrace x_\lambda: \lambda \in \Lambda \rbrace}$ and so we have $ \overline{\lbrace x_\lambda: \lambda \in \Lambda \rbrace}=S_X$. Let us fix $\mu\in \Lambda$. For all $\lambda \neq \mu$, we have $|1-x^*_{\lambda}(x_{\mu})|=|x^*_{\lambda}(x_{\lambda}-x_{\mu})|\leq \|x_{\lambda}-x_{\mu}\|$. This implies that $|x^*_{\lambda}(x_{\mu})| \geq 1-\|x_{\lambda}-x_{\mu}\|$ for all $\lambda\neq \mu$. By the part $(ii)$ of property quasi-$\alpha$, we get $1>\rho(\mu)\geq 1-\|x_{\lambda}-x_{\mu}\|$ for all $\lambda\neq \mu$. Since $\overline{\lbrace x_\lambda: \lambda \in \Lambda \rbrace}=S_X$, we can choose $(x_\lambda)$ such that $\lambda\neq \mu$ and $\|x_\lambda-x_\mu\|\to 0$ and obtain a contradiction. Hence, $(\R^n,\|\cdot\|_2)$ does not satisfy property quasi-$\alpha$.
\end{proof}

 The following remarks together with the diagram that follows it and Theorem \ref{cor1} below,  presents an overall summary of relationships between these different properties.
\begin{remark} \label{recap} The following assertions hold.

$(a)$ $($Property $\alpha$ or uniformly convex space$)$  $\Longrightarrow n.u.s.e. \Longrightarrow  n.u.s.$ 

$(b)$ $ n.u.s. \not \Longrightarrow  n.u.s.e.$

$(c)$ $ n.u.s.e. \not \Longrightarrow  \textnormal{quasi-}\alpha$. In consequence, $ n.u.s. \not \Longrightarrow  \textnormal{quasi-}\alpha$.

$(d)$ $\textnormal{two-dimensional space} \not \Longrightarrow   n.u.s.e.$

$(e)$ $ \textnormal{two-dimensional space} \not \Longrightarrow \textnormal{quasi-}\alpha$.

$(f)$ two-dimensional space $\Longrightarrow$ $n.u.s$. More generally, every finite dimentional space $X$ such that  $\textnormal{ext}(B_X)$ is closed  $\Longrightarrow$ $n.u.s$.
\end{remark}
\begin{proof} For $(a)$, the fact that property $\alpha$ implies property $n.u.s.e.$ is easy to see. The fact that a uniformly convex space has the property $n.u.s.e.$ follows directely from  \cite[Lemma 1]{Pbj}. The fact that the property $n.u.s.e.$ implies the property $n.u.s.e.$ follows immediately from the definitions. For the part $(b)$, we know that every finite dimentional space has property $n.u.s.$ by Proposition \ref{finite}, whereas Example ~\ref{Nfinite} gives an example of a $2$-dimensional space that does not have the property $n.u.s.e.$ For part $(c)$, it is shown in $(a)$ that every uniformly convex space  has property $ n.u.s.e.$, but the Euclidean space $(\R^n,\|\cdot\|_2)$ does not satisfy property quasi-$\alpha$ by Proposition ~\ref{quasi1}. The part $(d)$ follows from Example ~\ref{Nfinite} and the part $(e)$ is a consequence of Proposition \ref{quasi1}. Finally,  the part $(f)$ is given by Proposition ~\ref{finite}.
\end{proof}
As stated in the introduction,  we obtain the following extension of a Lindenstrauss result in \cite[Proposition 1]{Lj}, replacing the property $n.u.s.e.$ by the property $n.u.s.$ involving the stronger property $(\sigma A)$. 
\begin{theorem} \label{cor1} A Banach space with the property $n.u.s.$ has the property $(\sigma A)$ (in particular, it has the property $(A)$).
\end{theorem}

\begin{proof} Since $X$ has property $n.u.s.$ there exists a bounded symmetric subset $C$ with the $USP$ such that $B_X=\overline{\textnormal{co}}(C)$. Replacing $C$ by $\overline{C}$, we can assume without loss of generality that $C$ is closed, since $C$ has the $USP$ if and only if $\overline{C}$ has the $USP$. We apply directly Proposition \ref{Gprince} with $S=0$, observing that for every bounded linear operator $T\in \mathcal{L}(X,Y)$, we have $\|T\|=\sup_{x\in C}\|T(x)\|$ (that is, $C$ is a norming set).
\end{proof}
Uniformly convex spaces together with the spaces given in Example \ref{Exnus} have  the property $(\sigma A)$.

In the following diagram, the arrow ``$\longrightarrow$" means ``$\Longrightarrow$" and $RNP$ means ``Radon–Nikodym property":

\begin{tikzpicture}[every text node part/.style={align=center}]
\node(comp) at (4,5)[rectangle,draw,text width=3cm,text centered] {Property $(A)$};
\node(sigma) at (2,3.5)[rectangle,draw] {Property $(\sigma A)$};
\node(desk) at (2,2.4)[rectangle,draw] {Property $n.u.s.$};
\node(lap)  at (-1,-0.5)[rectangle,draw] {Uniformly convex space}; 
\node(HPl)  at (-1.7,1)[rectangle,draw] {Finite dimensional space $X$ \\ such that $\textnormal{ext}(B_X)$ is closed};     
\node(IBM)  at (2.6,1)[rectangle,draw] {Property $n.u.s.e.$\\ (Lindenstrauss \cite{Lj})};  
\node(IBMa)  at (4,-0.5)[rectangle,draw] {Property $\alpha$ \\ (Schachermayer \cite{Sch})};
\node(DELL) at (6.2,1)[rectangle,draw] {Property quasi-$\alpha$\\(Choi-Song \cite{CS})};   
\node(HPr)  at (7,2.4)[rectangle,draw] {$RNP$\\ Bourgain \cite{Bj}};        
                      
\draw[<-] (comp) -- (sigma); 
\draw[<-] (sigma) -- (desk);
\draw[<-] (desk) -- (HPl); 
\draw[<-] (desk) -- (IBM); 
\draw[<-] (IBM) -- (IBMa); 
\draw[<-] (DELL) -- (IBMa); 
\draw[<-] (comp)  -- (DELL); 
\draw[<-] (IBM) -- (lap); 
\draw[<-] (comp)  -- (HPr); 
\end{tikzpicture}


\section{Mappings attaining their norm} \label{Sd}
In this section we will consider two classes of non-linear mappings attaining their norm, namely the class of bounded multilinear mappings and the class of Lipschitz mappings.
\subsection{Bounded multilinear mappings  attaining their norm}\label{F4S}
For all $i\in \lbrace 1,..,n\rbrace$, let $(E_i, \|\cdot\|_i)$ and $W$ be Banach spaces. We denote $\mathcal{ML}(\prod_{i=1}^n E_i, W)$ the space of all bounded multilinear mappings from $\prod_{i=1}^n E_i$ into $W$ endowed with the norm
$$\|B\|=\sup \lbrace \|B(x_1,...,x_n)\|: x_i\in B_{E_i}, i\in \lbrace 1,...,n \rbrace \rbrace$$
for every $B\in \mathcal{ML}(\prod_{i=1}^n E_i, W)$. We say that a bounded multilinear mapping $B\in \mathcal{ML}(\prod_{i=1}^n E_i, W)$ is  {\it norm attaining} if the supremum deﬁning $\|B\|$ is actually a maximum. We prove in Corollary \ref{cor03} below that the set of all bounded multilinear mappings {\it norm attaining} is the complement of a $\sigma$-porous subset of $\mathcal{ML}(\prod_{i=1}^n E_i, W)$,  whenever $E_i$ has the property $n.u.s.e.$ for all $i\in \lbrace 1,...,n \rbrace$. We need the following lemma.

\begin{lemma} \label{Bilneaire}
For all $i\in \lbrace 1,..,n\rbrace$, let $(E_i, \|\cdot\|_i)$ be a Banach space and $S_i\subset S_{E_i}$ be a set of uniformly strongly exposed points of $B_{E_i}$ with a modulus $\omega_{S_i}(\cdot)$. Then, $(\prod_{i=1}^n S_i, \|\cdot\|)$ has the $USP$ in $E:=\prod_{i=1}^n E_i$ with the modulus $$\Delta(\varepsilon)=\frac{1}{n}\min \lbrace \omega_{S_i}(\frac{\varepsilon}{n}): i\in \lbrace 1,...,n \rbrace \rbrace, \forall \varepsilon>0.$$ (We take $\|(x_1,...,x_n)\|=\sum_{i=1}^n \|x_i\|_i$ or any other norm equivalent to this one).
\end{lemma}
\begin{proof} For all $i\in \lbrace 1,..,n\rbrace$, since $S_i$ is a set of uniformly strongly exposed points with a modulus $\omega_{S_i}(\cdot)$, there exists a family $\mathcal{F}_{S_i}:=\lbrace f_{i, t} \in S_{E^*_i}: t\in S_i\rbrace$ such that $f_{i,t}(t)=1$ for all $t\in S_i$ and :
\[ \forall t' \in B_{E_i}, \forall t\in S_i: \|t-t'\| \geq \varepsilon \Longrightarrow 1 -  \omega_{S_i}(\varepsilon) \geq  \langle f_{i,t}, t' \rangle.\]
We are going to prove that $C:=\prod_{i=1}^n S_i$ has the $USP$ in $E=\prod_{i=1}^n E_i$ with the modulus $\Delta(\cdot)$ and the family 
$$\mathcal{G}_{C}:=\left \lbrace g_{(x_1,...,x_n)}: (t_1,...,t_n)\mapsto \frac{1}{n}\sum_{i=1}^n f_{i,x_i}(t_i) \right\rbrace_{ (x_1,...,x_n)\in \prod_{i=1}^n S_i}\subset B_{E^*}.$$ 
Indeed, let $(x_1,...,x_n), (y_1,...,y_n)\in C$ and suppose that $\sum_{i=1}^n \|x_i-y_i\|_i\geq \varepsilon$. Then, necessarily there exists some $i_0\in \lbrace 1,..., n\rbrace$ such that  $\|x_{i_0}-y_{i_0}\|_{i_0}\geq \frac{\varepsilon}{n}$.  Thus, we have
\begin{eqnarray*} \label{EQt1} \langle f_{i_0,x_{i_0}}, x_{i_0} \rangle - \omega_{S_{i_0}}(\frac{\varepsilon}{n})= 1  - \omega_{S_{i_0}}(\frac{\varepsilon}{n}) \geq  \langle f_{i_0,x_{i_0}}, y_{i_0} \rangle.
\end{eqnarray*}
On the other hand, for all $z\in B_{E_i}$ with $i\in \lbrace 1,...,n \rbrace \setminus \lbrace i_0\rbrace$, we have
\begin{eqnarray*} \label{EQt2}
\langle f_{i,x_i}, x_i \rangle= 1\geq  \langle f_{i,x_i}, z\rangle.
\end{eqnarray*}
It follows that,
\begin{eqnarray*}  \sum_{i=1}^n \langle f_{i,x_i}, x_i \rangle  - n \Delta(\varepsilon) \geq \sum_{i=1}^n \langle f_{i,x_i}, x_i \rangle  - \omega_{S_{i_0}}(\frac{\varepsilon}{n}) \geq  \sum_{i=1}^n \langle f_{i,x_i}, y_i \rangle.
\end{eqnarray*}
Thus, 
\begin{eqnarray*}  g_{(x_1,...,x_n)}(x_1,...,x_n) -  \Delta(\varepsilon) \geq  g_{(x_1,...,x_n)}(y_1,...,y_n).
\end{eqnarray*}
This shows that $C:=\prod_{i=1}^n S_i$ has the $USP$ in $E=\prod_{i=1}^n E_i$ with the modulus $\Delta(\cdot)$.
\end{proof}
The  paper \cite{DGKLM} could be consulted for results concerning norm attaining multilinear mappings.
\begin{corollary} \label{cor03} Let $E_i$ be a Banach space having the property $n.u.s.e.$  for all $i\in \lbrace 1,..., n \rbrace$  and $W$ be any Banach space. Let $\mathcal{Z}$ be any closed subspace of $\mathcal{ML}(\prod_{i=1}^n E_i, W)$ containing the rank-one operators from $\prod_{i=1}^n E_i$ into $W$. Then, the following assertions hold.

$(i)$ The set of all bounded multilinear mappings $B\in \mathcal{Z}$ {\it norm attaining}, is the complement of a $\sigma$-porous subset of $\mathcal{Z}$.

$(ii)$ If we denote $S_i$ a set of uniformly strongly exposed points of $B_{E_i}$ with a modulus $\omega_{S_i}(\cdot)$, for all $i\in \lbrace 1,..., n \rbrace$, then we have the following quantitative  Bishop-Phelps-Bollobás property: for every $\varepsilon>0$, there exists $$\lambda(\varepsilon):=\frac{\varepsilon}{8n}\min \lbrace \omega_{S_i}(\frac{\varepsilon}{4n}): i\in \lbrace 1,...,n \rbrace \rbrace >0$$ such that for every $T\in \mathcal{Z}$, $\|T\|=1$ and every $(x_1,...,x_n)\in \prod_{i=1}^n S_i$ satisfying $\|T(x_1,...,x_n)\|>1-\lambda(\varepsilon),$ there exists $B \in \mathcal{Z}$ and $(\overline{x}_1,...,\overline{x}_n)\in \prod_{i=1}^n S_i$ such that: 

$(a)$ $\|B\|=\|B(\overline{x}_1,...,\overline{x}_n)\|=1$ and $B$ attains $\|\cdot\|^s$-strongly its sup-norm on $\prod_{i=1}^n S_i$ at $(\overline{x}_1,...,\overline{x}_n)$.

$(b)$ $\sum_{i=1}^n \|\bar{x}_i-x_i\|<\varepsilon \textnormal{ and } \|T-B\|<\varepsilon.$
\end{corollary}
\begin{proof} From the assumption, for each $i\in \lbrace 1,..., n \rbrace$ there exists a set  $S_i\subset S_E$ of uniform strongly exposed points satisfying $B_{E_i}=\overline{\textnormal{co}}(S_i)$. It is easy to see that for every $B\in\mathcal{ML}(\prod_{i=1}^n E_i, W)$ we have
$$\|B\|= \sup \left \lbrace \|B(x_1,...,x_n)\|: (x_1,...,x_n)\in \prod_{i=1}^n  S_i \right \rbrace.$$
By lemma \ref{Bilneaire} the set $\prod_{i=1}^n  S_i$ has the $USP$ in $\prod_{i=1}^n E_i$. Thus, we apply Theorem  \ref{Gprince1} with the set $C=\prod_{i=1}^n  S_i$ and by considering $\mathcal{Z}$ as a subset of $\mathcal{C}_b(C, W)$ (note that $\|B\|=\|B\|_{\infty}$, in this case). 
\end{proof}
\begin{remark} If we assume that $E_i$ is a uniformly convex Banach space with a modulus of uniform convexity $\delta_i(\cdot)$, then we can take $S_i=S_{E_i}$ and $\omega_{S_i}(\cdot)=\delta_i(\cdot)$ for all $i\in \lbrace 1,..., n \rbrace$ in the above corollary.
\end{remark}
\subsection{Lipschitz mappings attaining their norm} \label{FS}
Let $Y$ be a Banach space and $(M,d)$ be a pointed metric space, that is, a metric space in which we distinguish an element, called 0. By $\textnormal{Lip}_0(M,Y)$ we denote the Banach space of all $Y$-valued Lipschitz maps on $M$ that vanish at $0$, equipped with the norm:
$$\|g\|_L:=\sup_{x, y \in M: x\neq y} \frac{\|g(x)-g(y)\|}{d(x,y)}; \hspace{2mm} \forall g\in \textnormal{Lip}_0(M,Y).$$

When $Y=\R$, we simply denote $\textnormal{Lip}_0(M)$ instead of $\textnormal{Lip}_0(M,\R)$. The evaluation at $x\in M$ (or the Dirac mass associated to the point $x\in M$) is $\delta_x : g \mapsto g(x)$; $g\in \textnormal{Lip}_0(M)$. For each $x\in M$, $\delta_x\in (\textnormal{Lip}_0(M), \|\cdot\|_L)^*$ and the subspace $\overline{\textnormal{span}}\lbrace \delta_x: x\in M\rbrace$ of $(\textnormal{Lip}_0(M), \|\cdot\|_L)^*$ is denoted by $\mathcal{F}(M)$ and is called {\it Lipschitz-free} space of $M$ (for more details on Lipschitz-free space we refeer to \cite{GK}, \cite{W}, \cite{AE}). The Lipschitz-free space of $M$ is the predual of $\textnormal{Lip}_0(M)$, that is, $\mathcal{F}(M)^*=\textnormal{Lip}_0(M)$ as Banach spaces. It is easy to see that $\|\delta_x- \delta_y\|=d(x,y)$ for all $x,y \in M$. In other words, $\delta$ is an (non linear) isometry from $M$ into $\mathcal{F}(M)$. It is well-known that $\textnormal{Lip}_0(M,Y)$ is isometrically identiﬁed with the space of bounded linear operators $\mathcal{L}(\mathcal{F}(M), Y)$ by the identification $[f: M\to Y] \to [\hat{f} : \mathcal{F}(M) \to Y]$, where $\hat{f}$ is the linear continuous map defined by $\hat{f}(\delta_x)=\langle \hat{f}, \delta_x\rangle=f(x)$ for all $x\in M$.

For every $x,y\in M$ such that $x\neq y$, we denote $$\mathfrak{m}_{x,y}:=(\delta_x -\delta_y)/d(x,y).$$ 
We denote $\textnormal{Mol}(M):= \lbrace \mathfrak{m}_{x,y}: x, y \in M; x\neq y\rbrace$ the set of {\it ``molecules"} which is a subset of the unit sphere $S_{\mathcal{F}(M)}$ of the Lipschitz-free space. Note that, since $\textnormal{Mol}(M)$
 is balanced and norming for $\textnormal{Lip}_0(M,Y)$, a straightforward application of Hahn-Banach theorem implies that
$$B_{\mathcal{F}(M)}=\overline{\textnormal{co}}(\textnormal{Mol}(M)).$$
For every Lipschitz functional $f\in\textnormal{Lip}_0(M,Y)$, the action of $f$ at an element $\mathfrak{m}_{x,y}\in \textnormal{Mol}(M)$ is $\langle \hat{f},\mathfrak{m}_{x,y} \rangle=(f(x)-f(y))/d(x,y)\in Y$. With these notations we can write the norm of $f\in\textnormal{Lip}_0(M,Y)$, equivalently of $\hat{f}\in \mathcal{L}(\mathcal{F}(M), Y)$,  as follows:
$$\|f\|_L=\|\hat{f}\|=\sup_{\mathfrak{m}_{x,y}\in \textnormal{Mol}(M)}\|\langle \hat{f},\mathfrak{m}_{x,y}\rangle\|.$$

We say that $f$ attains its norm in the {\it ``strong sense"} if there exists a molecule $\mathfrak{m}_{x,y}$ such that $$\|f\|_L=\|\langle \hat{f}, \mathfrak{m}_{x,y}\rangle\|=\frac{\|f(x)-f(y)\|}{d(x,y)},$$ in other words, if the above supremum is attained at some point of $\textnormal{Mol}(M)$. The set of all Lipschitz functional $f \in \textnormal{Lip}_0(M,Y)$ attaining their norm at some point of the subset $\textnormal{Mol}(M)$ is denoted  $\textnormal{SNA}(M,Y)$ (see \cite{CCGMR, KMS}).  A consequence of the Bishop-Phelps theorem, stats that the set of real-valued Lipschitz functional $g\in\textnormal{Lip}_0(M)$ attaining their norm at some point of the closed unit ball $B_{\mathcal{F}(M)}$, is norm dense in $\textnormal{Lip}_0(M)$. Unfortunately it turns out that the set of Lipschitz functional attaining their norm at some point of the subset $\textnormal{Mol}(M)$, cannot be norm dense even in the one-dimentional case $M=\R$ (see for instance \cite{CCGMR}, \cite{KMS}).  

In \cite[Proposition 3.3 \& Corollary 3.8]{CCGMR}, the authors show that if $B_{\mathcal{F}(M)}$ is the closed convex hull of a set of uniformly strongly exposed points (equivalently, if $\mathcal{F}(M)$ has the property $n.u.s.e$), then  $\textnormal{SNA}(M,Y)$ is norm dense in $\textnormal{Lip}_0(M,Y)$. We extend these results in Corollary \ref{cor2} below  in the following way:  the density of $\textnormal{SNA}(M,Y)$ in $\textnormal{Lip}_0(M,Y)$ is replaced by the stronger fact that,  $\textnormal{SNA}(M,Y)\cap \mathcal{Z}$ is the complement of a $\sigma$-porous subset of $\mathcal{Z}$ for every closed subspace $\mathcal{Z}$ of $\textnormal{Lip}_0(M,Y)$ containing the sets $\textnormal{Lip}_0(M,\R)y$ for every $y\in Y$ (corresponding of the set of rank-one operators in $\mathcal{L}(\mathcal{F}(M), Y)$). 

\begin{corollary} \label{cor2} Let $(M,d)$ be a pointed metric space. Suppose that  $\mathcal{F}(M)$  has the property $n.u.s.e.$ Then,  $\textnormal{SNA}(M,Y)\cap \mathcal{Z}$ is the complement of a $\sigma$-porous subset of $\mathcal{Z}$ for every closed subspace $\mathcal{Z}$ of $\textnormal{Lip}_0(M,Y)$ containing the sets $\textnormal{Lip}_0(M,\R)y$ for every $y\in Y$. In particular $\textnormal{SNA}(M,Y)$ is the complement of a $\sigma$-porous subset of $\textnormal{Lip}_0(M,Y)$. 
\end{corollary}
\begin{proof} By definition $\mathcal{F}(M)$  has the property $n.u.s.e.$ if $B_{\mathcal{F}(M)}=\overline{\textnormal{co}}(S)$ for some symmetric set $S\subset S_{\mathcal{F}(M)}$ of uniformly strongly exposed points. From \cite[Proposition 1.1]{CCGMR}, we have $S\subset \textnormal{Mol}(M)$. Since the set $\textnormal{Mol}(M)$ is closed, we get $\overline{S}\subset \textnormal{Mol}(M)$. Now, since $(S,\|\cdot\|)$ and also $(\overline{S},\|\cdot\|)$ have in particular the $USP$, we use Proposition ~\ref{Gprince} with the set $C=\overline{S}$ and the space $\mathcal{L}(\mathcal{F}(M),Y)\simeq \textnormal{Lip}_0(M,Y)$ to obtain that for every closed subspace $\mathcal{Z}$ of $\mathcal{L}(\mathcal{F}(M),Y)$ containing the set of rank-one operators, the set 
\begin{eqnarray*}
\lbrace T\in \mathcal{Z}: T \textnormal{ attains } \|\cdot\|^s\textnormal{-strongly its sup-norm on } C\rbrace,
\end{eqnarray*}
is the complement of a $\sigma$-porous set of $\mathcal{Z}$. This implies that, $\textnormal{SNA}(M,Y)\cap \mathcal{Z}$ is the complement of a $\sigma$-porous subset of $\mathcal{Z}$ for every closed subspace $\mathcal{Z}$ of $\textnormal{Lip}_0(M,Y)$ containing the sets $\textnormal{Lip}_0(M,\R)y$ for every $y\in Y$ (corresponding of the set of rank-one in $\mathcal{L}(\mathcal{F}(M), Y)$). In particular with $\mathcal{Z}=\textnormal{Lip}_0(M,Y)$ we obtain that $\textnormal{SNA}(M,Y)$ is the complement of a $\sigma$-porous subset of $\textnormal{Lip}_0(M,Y)$. 
\end{proof}
We refer to \cite{CCGMR} for a number of examples where $\mathcal{F}(M)$ has the property $n.u.s.e.$ Notice that the above corollary is also true if we replace the property $n.u.s.e.$ by the existence of a symmetric subset $C$ of $\textnormal{Mol}(M)$ such that $(C,\|\cdot\|)$ has the $USP$ and satisfies $B_{\mathcal{F}(M)}=\overline{\textnormal{co}}(C)$ (which means in particular that $\mathcal{F}(M)$ has the property $n.u.s.$).

\subsection*{Acknowledgements}
This research has been conducted within the FP2M federation (CNRS FR 2036) and  SAMM Laboratory of the University Paris Panthéon-Sorbonne.

\bibliographystyle{amsplain}

\end{document}